\documentclass[11pt]{article}

\usepackage{amsmath,amssymb,amsthm,mathtools,bm}
\usepackage[margin=1in]{geometry}
\usepackage[authoryear,round]{natbib}
\usepackage{url}
\usepackage{enumitem}
\usepackage{booktabs}
\usepackage{microtype}
\usepackage{lineno}

\newtheorem{theorem}{Theorem}[section]
\newtheorem{proposition}[theorem]{Proposition}
\newtheorem{lemma}[theorem]{Lemma}
\newtheorem{corollary}[theorem]{Corollary}
\theoremstyle{remark}
\newtheorem{remark}[theorem]{Remark}

\newcommand{\R}{\mathbb R}

\newcommand{\Ep}{\mathbb E}
\newcommand{\ind}{\mathbf 1}
\newcommand{\thet}{\boldsymbol\theta}
\newcommand{\thetao}{\boldsymbol\theta_0}
\newcommand{\Xv}{\mathbf X}
\newcommand{\xv}{\mathbf x}

\title{Fisher Information and Differentiability for $r$-Largest Generalized
Extreme-Value Models}
\author{Jihong Park$^{1}$, Yire Shin$^{2,*}$, Jeong-Soo Park$^{2}$\\[1ex]
\small $^1$Department of Mathematics and Statistics, Chonnam National University, Gwangju 61186, Korea\\
\small $^2$Department of Statistics, Chonnam National University, Gwangju 61186, Korea\\
\small $^*$Corresponding author: shinyire87@gmail.com}
\date{}

\begin{document}
\maketitle

\begin{abstract}
	The $r$ largest observations of a block form a dependent random vector whose
	support moves with all three parameters of the $r$-largest generalized
	extreme-value ($r$GEV) model. For a fixed number $r$ of order statistics, we show
	that the model is differentiable in quadratic mean if and only if $k_0<1/2$, where
	$k=-\gamma$ is the shape parameter; the threshold is thus the same as for block
	maxima, because only the largest or the smallest retained order statistic can
	reach the moving boundary. In this range the Fisher information is positive
	definite and decomposes into the block-maximum information plus contributions of
	the lower order statistics, and it equals the negative expected Hessian, which we
	prove without using any estimator. The model is then locally asymptotically
	normal, whereas for $k_0\ge1/2$ root-$n$ local asymptotic normality fails.
	Numerically, using the five largest order statistics instead of the block
	maximum alone lowers the efficient asymptotic standard error by about $45\%$ for the
	shape parameter and by $20$--$24\%$ for the 100-year return level in
	heavy-tailed cases ($k_0\le-0.2$), but by only about $2\%$ for the return level
	when $k_0=0.4$.
\end{abstract}



\section{Introduction}
\label{sec:introduction}

Likelihood inference for the three-parameter generalized extreme-value (GEV)
distribution is not a routine application of regular theory, because the
support depends on the parameters. \citet{prescott1980maximum} derived the
information matrix, \citet{Smith1985} showed for families with moving
endpoints that the local behavior of the density at the boundary decides whether
root-$n$ normal theory holds, and \citet{ChengTraylor1995} surveyed such
non-regular problems. For the GEV family, \citet{buchersegers2017maximum}
proved differentiability in quadratic mean (DQM) exactly for $\gamma>-1/2$, which
is $k<1/2$ in the convention $k=-\gamma$ used here, and noted that local
asymptotic normality (LAN) follows. A complementary line of work
treats the GEV law as an approximation to block maxima
\citep{dombry2015existence,dombry2019maximum}.

The $r$GEV model, which uses the $r$ largest observations of each block
\citep{smith1986extreme,tawn1988extreme}, is widely used for environmental
extremes \citep{soares2004application,an2007r,wang2008downscaling,hamdi2014extreme,muhammed2017variations,e2020change}.
The model has been extended in several directions
\citep{shin2023modeling,ShinPark2024Serra,ShinPark2025SciRep,Busaba2025CSAM,Busababodhin2026,silva2022dynamic}
and is available in R \citep{ShinParkRpackageevmr,bader2020package}. The
choice of $r$ is studied by \citet{bader2017automated} and \citet{ShinParkPark2026}.
Related recent work concerns serial dependence \citep{BucherHaufs2026} and the
dependence among the largest order statistics \citep{heinonen7451430dependence}.
The Fisher information contained in ordered data has been studied for
families with fixed support, for example for generalized order statistics, which
include ordinary order statistics and record values \citep{BurkschatCramer2012}.

A regularity theory for the exact three-parameter $r$GEV likelihood, however,
has not been available, and it does not follow from the case $r=1$. The
observation is an $r$-dimensional random vector whose coordinates are dependent
through their ordering, and its joint support is described by $r$ endpoint
inequalities that move with all three parameters. Whether the boundary of this
multivariate support can make the model less regular than the block-maximum model,
and how the information carried by the lower order statistics is structured, are
the questions addressed here.

This paper answers the first question of such a theory, which concerns the model
itself rather than any estimator. We prove that the fixed-$r$ $r$GEV family is
DQM if and only if $k_0<1/2$ (Theorem~\ref{thm:dqm}); that its Fisher
information decomposes sequentially and is positive definite
(Proposition~\ref{prop:info-pd}); that the information identity holds
(Theorem~\ref{thm:info-identity}); and that the model is locally asymptotically
normal, with the corresponding lower bound, while local asymptotic normality
fails for $k_0\ge1/2$ (Corollary~\ref{cor:lan}, Remark~\ref{rem:lan-failure}).
Thus the regularity threshold does not depend on $r$: the leading maximum
determines the boundary of regular behavior, while lower order statistics can
only add information.

For $r=1$, the DQM characterization and LAN are due to
\citet{buchersegers2017maximum}, and our proof of sufficiency follows their
moving-support argument. The new elements are the following:
(i) for every fixed $r$, only $X^{(1)}$ (if $k>0$) or $X^{(r)}$ (if $k<0$) can
reach the boundary of the joint support, so that the threshold is unchanged;
(ii) the martingale structure of the conditional scores, which yields the
sequential information decomposition;
(iii) a proof of the information identity that uses no estimator and no
differentiation under a moving integral; and
(iv) an explicit proof that the block-maximum information $I_1$ is positive
definite; and (v) a numerical evaluation of the resulting information gain
(Section~\ref{sec:numerics}): from $r=1$ to $r=5$ the efficient asymptotic
standard error decreases by about $45\%$ for the shape parameter and by
$20$--$24\%$ for the 100-year return level in heavy-tailed cases, whereas for
$k_0>0$ the latter gain is small. The asymptotic theory of the maximum likelihood estimator and of
likelihood-based inference, which builds on these results, is the subject of
ongoing work. 

Section~\ref{sec:model} introduces the model,
Section~\ref{sec:results} states the results, Section~\ref{sec:numerics} evaluates
the information gain, and Section~\ref{sec:proofs} contains the main proofs; the
remaining technical lemmas and the moving-strip part of the DQM proof are in the
appendices.

\section{The fixed-$r$ $r$GEV model}
\label{sec:model}

Let $r\ge1$ be fixed, and let $\Xv_1,\ldots,\Xv_n$ be independent copies of
$\Xv=(X^{(1)},\ldots,X^{(r)})$, the $r$ largest order statistics of a block.
The order statistics within a block are dependent through their ordering, and
the blocks are assumed independent. Within blocks, the $r$GEV form presupposes
that the retained values are not clustered, which in applications is achieved
by suitable declustering \citep{tawn1988extreme}.

Let $\thet=(\mu,\sigma,k)$ with $\sigma>0$, put
$w_{\thet}(x)=1-k(x-\mu)/\sigma$ and, for $k\ne0$,
$T_{\thet}(x)=w_{\thet}(x)^{1/k}$, with the continuous limit
$T_{(\mu,\sigma,0)}(x)=\exp\{-(x-\mu)/\sigma\}$. The GEV distribution function is
$F_{\thet}(x)=\exp\{-T_{\thet}(x)\}$.

\begin{remark}[Shape-parameter convention]
\label{rem:shape-convention}
This is the convention of \citet{prescott1980maximum} and \citet{tawn1988extreme}.
The convention $\exp[-\{1+\gamma(x-\mu)/\sigma\}^{-1/\gamma}]$ of, e.g.,
\citet{coles2001introduction} is related to it by
\begin{equation}
\gamma=\xi=-k,
\label{eq:shape-sign-conversion}
\end{equation}
so $\gamma>-1/2$ corresponds to $k<1/2$, and $k>0$ to a finite upper endpoint
$\omega(\thet)=\mu+\sigma/k$.
\end{remark}

The joint density of the $r$ largest order statistics is
\begin{equation}
 p_{r,\thet}(x_1,\ldots,x_r)
 =\sigma^{-r}\exp\{-T_{\thet}(x_r)\}
 \prod_{j=1}^r T_{\thet}(x_j)^{1-k}
 \label{eq:rgev-density}
\end{equation}
on the support
$S_{r,\thet}=\{\xv\in\R^r:x_1>\cdots>x_r,\ w_{\thet}(x_j)>0,\ j=1,\ldots,r\}$. We
write $\ell_{r,\thet}=\log p_{r,\thet}$ on $S_{r,\thet}$, with
$\ell_{r,\thet}=-\infty$ outside it, and $P_{r,\thet}$, $\Ep_{\thet}$ for the
corresponding law and expectation. At the true parameter $\thetao$, the
variables $T_j=T_{\thetao}(X^{(j)})$ are the first $r$ arrival times of a
unit-rate Poisson process \citep{Weissman1978,Resnick1987}:
\begin{equation}
T_1,\ T_2-T_1,\ldots,T_r-T_{r-1}
\stackrel{\rm ind}{\sim}\operatorname{Exp}(1),
\label{eq:arrival}
\end{equation}
so $T_j\sim\operatorname{Gamma}(j,1)$; this law does not depend on $\thetao$. In
terms of $T_j=T_{\thet}(x_j)$,
\begin{equation}
\ell_{r,\thet}
=-r\log\sigma-T_r+(1-k)\sum_{j=1}^r\log T_j.
\label{eq:loglik-T}
\end{equation}
Let $D_{\thet}(x)=\nabla_{\thet}\log T_{\thet}(x)$. For $k\ne0$,
\begin{equation}
D_{\thet}(x)=
\begin{pmatrix}
T_{\thet}(x)^{-k}/\sigma\\[1mm]
\{T_{\thet}(x)^{-k}-1\}/(k\sigma)\\[1mm]
-\log T_{\thet}(x)/k-
\{T_{\thet}(x)^{-k}-1\}/k^2
\end{pmatrix},
\label{eq:D}
\end{equation}
with continuous limits at $k=0$, and, writing $D_j=D_{\thet}(x_j)$, the score is
\begin{equation}
\dot{\boldsymbol\ell}_{r,\thet}
=-\frac r\sigma\bm e_\sigma
-T_rD_r
+(1-k)\sum_{j=1}^rD_j
-\bm e_k\sum_{j=1}^r\log T_j.
\label{eq:joint-score}
\end{equation}
Put $\Theta_{1/2}=\{(\mu,\sigma,k):\sigma>0,\ k<1/2\}$.

\section{Main results}
\label{sec:results}

Let $s_{\thet}=p_{r,\thet}^{1/2}$, extended by zero outside $S_{r,\thet}$. The
family is DQM at $\thetao$ if there is a measurable $\dot{\boldsymbol\ell}_{\thetao}$
with
$\int\{s_{\thetao+\bm h}-s_{\thetao}-\tfrac12\bm h^T\dot{\boldsymbol\ell}_{\thetao}
s_{\thetao}\}^2d\xv=o(\|\bm h\|^2)$ as $\bm h\to\bm 0$
\citep[Section~7.2]{vanderVaart1998}. The definition does not require fixed
supports, which makes it the appropriate notion here.

\begin{theorem}[Differentiability in quadratic mean]
\label{thm:dqm}
For fixed $r\ge1$, the three-parameter $r$GEV family is
differentiable in quadratic mean at
$\thetao=(\mu_0,\sigma_0,k_0)$ if and only if $k_0<1/2$.
When this condition holds, the quadratic-mean score agrees
with the ordinary likelihood gradient $\dot{\boldsymbol\ell}_{r,\thetao}$ on the
support and vanishes outside it.
\end{theorem}

\begin{proposition}[Sequential Fisher-information decomposition]
\label{prop:info-pd}
For fixed $r\ge1$ and $k_0<1/2$, the Fisher information
$I_r(\thetao)=\Ep_{\thetao}\dot{\boldsymbol\ell}_{r,\thetao}\dot{\boldsymbol\ell}_{r,\thetao}^T$
is finite and positive definite. More precisely,
\[
I_r(\thetao)=I_1(\thetao)
+\sum_{s=2}^r\Ep_{\thetao}(\bm c_s\bm c_s^T)
\succeq I_1(\thetao)>0,
\]
where $\bm c_s$ denotes the score of the conditional
density of $X^{(s)}$ given $X^{(s-1)}$.
\end{proposition}

\begin{theorem}[Information identity]
\label{thm:info-identity}
For fixed $r\ge1$ and $k_0<1/2$,
\[
\Ep_{\thetao}\|\ddot{\boldsymbol\ell}_{r,\thetao}(\Xv)\|<\infty,
\qquad
-\Ep_{\thetao}\ddot{\boldsymbol\ell}_{r,\thetao}(\Xv)
=I_r(\thetao).
\]
\end{theorem}

Write $P^n_{\thet}$ for the law of $(\Xv_1,\ldots,\Xv_n)$ under $\thet$ and
$\bm\Delta_n=n^{-1/2}\sum_{i=1}^n\dot{\boldsymbol\ell}_{r,\thetao}(\Xv_i)$.

\begin{corollary}[Local asymptotic normality and lower bound]
\label{cor:lan}
Let $r\ge1$ be fixed and $k_0<1/2$.
\begin{enumerate}[label=(\roman*)]
\item For every sequence $\bm h_n\to\bm h$ in $\R^3$,
\[
\sum_{i=1}^n\log\frac{p_{r,\thetao+\bm h_n/\sqrt n}(\Xv_i)}{p_{r,\thetao}(\Xv_i)}
=\bm h^T\bm\Delta_n-\frac12\bm h^TI_r(\thetao)\bm h+o_p(1)
\]
under $P^n_{\thetao}$, and $\bm\Delta_n\rightsquigarrow N_3\{0,I_r(\thetao)\}$.
In particular, $P^n_{\thetao+\bm h_n/\sqrt n}$ and $P^n_{\thetao}$ are mutually
contiguous.
\item Let $\psi:\Theta_{1/2}\to\R^m$ be differentiable at $\thetao$, and let
$L:\R^m\to[0,\infty)$ be bowl-shaped. Then for every estimator sequence $\widehat\psi_n$,
\[
\sup_{\mathcal J}\liminf_{n\to\infty}\sup_{\bm h\in\mathcal J}
\Ep_{\thetao+\bm h/\sqrt n}
L\bigl[\sqrt n\{\widehat\psi_n-\psi(\thetao+\bm h/\sqrt n)\}\bigr]
\ge\Ep L(\bm Z),
\]
where $\mathcal J$ ranges over finite subsets of $\R^3$ and
$\bm Z\sim N_m\{0,\dot\psi_{\thetao}I_r(\thetao)^{-1}\dot\psi_{\thetao}^T\}$.
\end{enumerate}
\end{corollary}

\begin{remark}[Failure of contiguity for $k_0\ge1/2$]
\label{rem:lan-failure}
Let $k_0\ge1/2$, $a>0$ and $\thet_n=(\mu_0+a/\sqrt n,\sigma_0,k_0)$. Then
$P^n_{\thet_n}$ is not contiguous with respect to $P^n_{\thetao}$, so no expansion
of the form in Corollary~\ref{cor:lan}(i) can hold; the usual root-$n$ LAN
expansion fails along this location direction, and a different normalization
is needed in this range. The proof is in Section~\ref{sec:proof-lan}.
\end{remark}

Corollary~\ref{cor:lan}(ii) applies to the parameters and to return levels
$z_p(\thet)=\mu+(\sigma/k)(1-y_p^{\,k})$, $y_p=-\log(1-p)$. Since
$I_r(\thetao)\succeq I_1(\thetao)$, the lower bound based on the $r$ largest order
statistics is never larger than that based on block maxima alone, whereas the
regularity threshold $k_0<1/2$ is the same for every $r$.

\section{How much information do the lower order statistics add?}
\label{sec:numerics}

Proposition~\ref{prop:info-pd} shows that each additional order statistic adds
a positive semidefinite term to the Fisher information, and
Corollary~\ref{cor:lan}(ii) turns $I_r(\thetao)^{-1}$ into a lower bound for the
asymptotic risk of every estimator sequence. This section quantifies the gain.
Since the conditional scores $\bm c_s$ in Proposition~\ref{prop:info-pd} do not
depend on $r$,
\[
I_{r+1}(\thetao)-I_r(\thetao)=\Ep_{\thetao}(\bm c_{r+1}\bm c_{r+1}^T)\succeq0,
\]
so $I_r(\thetao)^{-1}$ decreases in $r$ in the positive semidefinite order. The
matrix $I_r(\thetao)$ does not depend on $\mu_0$, and with
$S=\operatorname{diag}(\sigma_0,\sigma_0,1)$ the matrix $SI_r(\thetao)S$ depends only
on $k_0$.

We consider two functionals: the shape parameter, $\psi(\thet)=k$, and the
return level, the $(1-p)$-quantile of the block maximum; with annual blocks,
$z_{0.01}$ is the 100-year return level:
\[
z_p(\thet)=
\begin{cases}
\mu+\sigma(1-y_p^{\,k})/k,&k\ne0,\\
\mu-\sigma\log y_p,&k=0,
\end{cases}
\qquad y_p=-\log(1-p),
\]
whose gradient is
\[
\begin{aligned}
\nabla z_p(\thet)&=
\Bigl(1,\ \frac{1-y_p^{\,k}}{k},\
-\frac{\sigma(1-y_p^{\,k})}{k^2}-\frac{\sigma y_p^{\,k}\log y_p}{k}\Bigr)^T
&&(k\ne0),\\
\nabla z_p(\thet)&=\bigl(1,\,-\log y_p,\,-\tfrac12\sigma(\log y_p)^2\bigr)^T
&&(k=0).
\end{aligned}
\]
For a differentiable $\psi$, the normal law in Corollary~\ref{cor:lan}(ii) has
standard deviation
\[
\{\dot\psi_{\thetao}I_r(\thetao)^{-1}\dot\psi_{\thetao}^T\}^{1/2}.
\]
We refer to this
quantity as the efficient asymptotic standard deviation of $\psi$. Its square is
the local asymptotic minimax variance bound, and by the convolution theorem
\citep[Theorem~8.8]{vanderVaart1998} it is also the smallest possible standard
deviation of the normal limit of $\sqrt n\{\widehat\psi_n-\psi(\thetao)\}$ for a
regular estimator sequence $\widehat\psi_n$. The corresponding efficient
asymptotic standard error of $\widehat\psi_n$ is this quantity divided by
$\sqrt n$, so relative reductions in $r$ are the same for both. It depends only
on $k_0$ for
$\psi=k$, and it is $\sigma_0$ times a function of $k_0$ and $p$ for $\psi=z_p$.

We computed $I_r(\thetao)$ from the sequential decomposition of
Proposition~\ref{prop:info-pd}: $I_1(\thetao)$ by one-dimensional numerical
integration against the $\operatorname{Exp}(1)$ law of $T_1$, and each
$\Ep_{\thetao}(\bm c_s\bm c_s^T)$ by two-dimensional numerical integration against
the joint law of $(T_{s-1},T_s)$ in \eqref{eq:arrival}, where, by
\eqref{eq:loglik-T} and the factorization \eqref{eq:sequential-factorization},
\[
\bm c_s=-\sigma_0^{-1}\bm e_\sigma+(1-k_0-T_s)D_{\thetao}(X^{(s)})
-\bm e_k\log T_s+T_{s-1}D_{\thetao}(X^{(s-1)}).
\]
The values of $I_1(\thetao)$ agree with the closed form of
\citet{prescott1980maximum}, and the conditional contributions agree with
Monte Carlo averages to within the Monte Carlo error.

\begin{table}[!htb]
\centering
\caption{Efficient asymptotic standard deviations
$\{\dot\psi_{\thetao}I_r(\thetao)^{-1}\dot\psi_{\thetao}^T\}^{1/2}$ for the shape
parameter $\psi=k$ and for the 100-year return level $\psi=z_{0.01}$, with
$\sigma_0=1$.}
\label{tab:asd}
\begin{tabular}{@{}r cccc c cccc@{}}
\toprule
& \multicolumn{4}{c}{shape $k$} && \multicolumn{4}{c}{return level $z_{0.01}$}\\
\cmidrule{2-5}\cmidrule{7-10}
$k_0$ & $r=1$ & $r=2$ & $r=3$ & $r=5$ && $r=1$ & $r=2$ & $r=3$ & $r=5$\\
\midrule
$-$0.5 & 0.968 & 0.775 & 0.675 & 0.559 && 60.43 & 54.18 & 50.53 & 45.75 \\
$-$0.4 & 0.913 & 0.731 & 0.635 & 0.526 && 39.57 & 35.73 & 33.45 & 30.44 \\
$-$0.3 & 0.858 & 0.685 & 0.595 & 0.492 && 25.94 & 23.62 & 22.22 & 20.34 \\
$-$0.2 & 0.803 & 0.640 & 0.555 & 0.458 && 16.99 & 15.64 & 14.80 & 13.65 \\
$-$0.1 & 0.747 & 0.594 & 0.514 & 0.423 && 11.09 & 10.34 & 9.86 & 9.18 \\
0.0 & 0.690 & 0.547 & 0.473 & 0.388 && 7.20 & 6.80 & 6.54 & 6.17 \\
0.1 & 0.633 & 0.500 & 0.431 & 0.353 && 4.61 & 4.42 & 4.30 & 4.11 \\
0.2 & 0.575 & 0.451 & 0.388 & 0.317 && 2.88 & 2.81 & 2.77 & 2.69 \\
0.3 & 0.515 & 0.402 & 0.345 & 0.281 && 1.72 & 1.71 & 1.70 & 1.68 \\
0.4 & 0.454 & 0.351 & 0.300 & 0.244 && 0.94 & 0.93 & 0.93 & 0.93 \\
\bottomrule
\end{tabular}
\end{table}

Table~\ref{tab:asd} reports the results for $k_0\in[-0.5,0.4]$. The efficient
asymptotic standard deviation for the shape parameter decreases substantially
with $r$: by $20$--$23\%$ from $r=1$ to
$r=2$ and by $42$--$46\%$ from $r=1$ to $r=5$, almost uniformly in $k_0$; the
reduction for $\sigma$ (not shown) is between $16\%$ and $66\%$ at $r=5$. The
reduction for the 100-year return level depends strongly on $k_0$: from $r=1$
to $r=5$ it is $20$--$24\%$ in the heavy-tailed cases $k_0\le-0.2$ (for example,
from $60.43$ to $45.75$ for $k_0=-0.5$, and from $16.99$ to $13.65$ for
$k_0=-0.2$), increasing with the heaviness of the tail, about
$14\%$ in the Gumbel case, and only $2\%$ for $k_0=0.4$. When $k_0>0$ the return level lies close to
the upper endpoint $\omega(\thetao)$; a plausible explanation is that information
about this region of the distribution comes mainly from the block maxima, so that
the lower order statistics sharpen the estimation of $k$ and $\sigma$ but add
little for $z_p$. These figures concern the exact $r$GEV model; when the model is
only an approximation for the $r$ largest observations of a block, larger $r$ may
also increase the bias \citep{smith1986extreme,BucherHaufs2026}, which is one of the
considerations behind data-based choices of $r$
\citep{bader2017automated,ShinParkPark2026}.

\section{Proofs}
\label{sec:proofs}

\subsection{A zero-extension lemma}
\label{sec:zero-ext}

The parameter-dependent support is handled repeatedly through the following
elementary fact. For fixed real numbers $x_1,\ldots,x_m$ put
\[
\pi_j(\thet)=\sigma-k(x_j-\mu),\qquad j=1,\ldots,m,
\]
so that $w_{\thet}(x_j)>0$ if and only if $\pi_j(\thet)>0$.

\begin{lemma}[Zero extension across a moving endpoint]
\label{lem:zero-extension}
Let $V\subset\{\thet:\sigma>0\}$ be open and convex and let
$O=\{\thet\in V:\pi_j(\thet)>0,\ j=1,\ldots,m\}$. Let $g:V\to\R$ and
$c\in\R$ satisfy:
\begin{enumerate}[label=(\roman*)]
\item $g$ is continuously differentiable on $O$;
\item $g\equiv c$ on $V\setminus O$;
\item for every $\thet^*\in V\cap\partial O$, $g(\thet)\to c$ and
$\nabla g(\thet)\to0$ as $\thet\to\thet^*$ with $\thet\in O$.
\end{enumerate}
Then $g$ is continuously differentiable on $V$, with $\nabla g=0$ on
$V\setminus O$, and for all $\thet_1,\thet_2\in V$,
\begin{equation}
|g(\thet_1)-g(\thet_2)|
\le\sup_{\vartheta\in[\thet_1,\thet_2]\cap O}\|\nabla g(\vartheta)\|\,
\|\thet_1-\thet_2\|,
\label{eq:zero-extension-mvt}
\end{equation}
where the supremum over the empty set is $0$.
\end{lemma}

\begin{proof}
The idea is to restrict $g$ to the segment joining $\thet_1$ and $\thet_2$
and to show that this segment crosses the moving boundary only finitely
often. Let $\thet(u)=\thet_1+u(\thet_2-\thet_1)$ and $\varphi(u)=g\{\thet(u)\}$,
$u\in[0,1]$. Since $k(u)$ and $\mu(u)$ are affine in $u$, each
$u\mapsto\pi_j\{\thet(u)\}=\sigma(u)-k(u)\{x_j-\mu(u)\}$ is a polynomial of
degree at most two. Unless it vanishes identically, it has at most two zeros,
so the set where it is positive is a union of at most three relatively open
intervals. Hence $I=\{u:\thet(u)\in O\}$ is a finite union of relatively open
intervals. If some $\pi_j\circ\thet$ vanishes identically, then $I=\emptyset$
and $\varphi\equiv c$, so there is nothing to prove.

The function $\varphi$ is continuous on $[0,1]$: it is continuous on $I$ by
(i), constant on $[0,1]\setminus I$ by (ii), and at an endpoint $u_0\notin I$
of a component of $I$ the point $\thet(u_0)$ belongs to $V\cap\partial O$, so
(iii) gives $\varphi(u)\to c=\varphi(u_0)$ as $u\to u_0$ within $I$. Moreover,
$\varphi$ is continuously differentiable except at the finitely many endpoints
of the components of $I$, with
$|\varphi'(u)|\le\|\nabla g\{\thet(u)\}\|\,\|\thet_2-\thet_1\|$ on $I$ and
$\varphi'=0$ outside the closure of $I$. This derivative is bounded, by (i),
(iii) and compactness of the segment. Applying the fundamental theorem of
calculus on each of the finitely many subintervals and adding the results
gives $|\varphi(1)-\varphi(0)|\le\int_0^1|\varphi'(u)|\,du$, which is
\eqref{eq:zero-extension-mvt}.

Let $\thet^*\in V\setminus O$. If $\thet^*$ is interior to $V\setminus O$,
$g$ is locally constant. Otherwise $\thet^*\in V\cap\partial O$, and
\eqref{eq:zero-extension-mvt} with $(\thet_1,\thet_2)=(\thet^*,\thet^*+\bm h)$
gives $|g(\thet^*+\bm h)-c|=o(\|\bm h\|)$ by (iii). Hence
$\nabla g(\thet^*)=0$, and (iii) also gives continuity of $\nabla g$ at
$\thet^*$.
\end{proof}

Lemma~\ref{lem:zero-extension} is used in the proof of DQM for $k_0<1/3$ (Lemma~\ref{lem:sqrt-smooth}); the same finite-crossing idea, applied along
coordinate segments, underlies Lemma~\ref{lem:ac-endpoint} below. It formalizes
a familiar heuristic: a function
that vanishes outside a moving region remains smooth across the moving
boundary, provided that both the function and its gradient vanish as the
boundary is approached from inside. It is a multivariate version of the
boundary argument in \citet[Lemma~B.5 and Appendix~C]{buchersegers2017maximum}.
Condition (iii) is a pointwise requirement for each fixed
$(x_1,\ldots,x_m)$; no uniformity in the observation is needed. Integrability
is always controlled separately, through explicit envelopes.

\subsection{Proof of Theorem~\ref{thm:dqm}}
\label{sec:proof-dqm}

\emph{Necessity.} Let $k_0\ge1/2$ and $\thet_t=(\mu_0+t,\sigma_0,k_0)$,
$t\downarrow0$. The upper endpoint moves from $\omega_0=\mu_0+\sigma_0/k_0$ to
$\omega_0+t$. Since the first marginal is the GEV distribution,
\begin{equation}
P_{\thet_t}\{p_{r,\thetao}=0\}
\ge P_{\thet_t}\{X^{(1)}\ge\omega_0\}
=1-\exp\Bigl[-\Bigl(\frac{k_0t}{\sigma_0}\Bigr)^{1/k_0}\Bigr],
\label{eq:dqm-necessity-mass}
\end{equation}
which is not $o(t^2)$ when $1/k_0\le2$. On $\{p_{r,\thetao}=0\}$ both
$s_{\thetao}$ and $\dot{\boldsymbol\ell}_{\thetao}s_{\thetao}$ vanish, so the DQM
integral is at least $\int_{\{p_{r,\thetao}=0\}}s_{\thet_t}^2\,d\xv
=P_{\thet_t}\{p_{r,\thetao}=0\}$, which would have to be $o(t^2)$. Hence DQM fails.

\emph{Sufficiency} is proved in Appendix~\ref{sec:S-dqm}. For $k_0<1/3$, the
square-root density is continuously differentiable in $\thet$ for every $\xv$ (Lemma~\ref{lem:sqrt-smooth}, via Lemma~\ref{lem:zero-extension}), and DQM follows
from Lemma~7.6 of \citet{vanderVaart1998}. For $1/3\le k_0<1/2$, a strip of width
$O(t)$ below the moving endpoint carries probability $o(t^2)$; this is where
$k_0<1/2$ enters, in agreement with \eqref{eq:dqm-necessity-mass}.

\subsection{The sequential Fisher-information decomposition}
\label{sec:proof-info}

Let $f_{\thet}$ be the GEV density. The joint density factorizes as
\begin{equation}
p_{r,\thet}(\xv)=f_{\thet}(x_1)\prod_{s=2}^rq_{s,\thet}(x_s\mid x_{s-1}),
\qquad
q_{s,\thet}(y\mid x)=\frac{f_{\thet}(y)}{F_{\thet}(x)}\ind(y<x),
\label{eq:sequential-factorization}
\end{equation}
where $q_{s,\thet}(\cdot\mid x)$ is the conditional density of $X^{(s)}$ given
$X^{(s-1)}=x$. With
$\bm c_s=\nabla_{\thet}\log q_{s,\thet}(X^{(s)}\mid X^{(s-1)})|_{\thet=\thetao}$,
\begin{equation}
\dot{\boldsymbol\ell}_{r,\thetao}
=\dot{\boldsymbol\ell}_{1,\thetao}(X^{(1)})+\sum_{s=2}^r\bm c_s .
\label{eq:score-decomposition}
\end{equation}

The key step is the conditional score identity. For $k<0$ the lower endpoint of
the support of $q_{s,\thet}(\cdot\mid x)$ moves with $\thet$, so differentiating
$\int q_{s,\thet}(y\mid x)\,dy=1$ under the integral sign requires separate
justification. We verify the identity directly on the arrival scale, which
treats all values of $k$ at once. Write $\lambda_{\thet}(y)=-\partial_yT_{\thet}(y)$, so that
$\log q_{s,\thet}(y\mid x)=\log\lambda_{\thet}(y)-T_{\thet}(y)+T_{\thet}(x)$ for
$y<x$, and let $G(u)=\nabla_{\thet}T_{\thet}(y)|_{\thet=\thetao}$ at the point $y$
with $T_{\thetao}(y)=u$. The chain rule gives
$G'(u)=\nabla_{\thet}\log\lambda_{\thet}(y)|_{\thet=\thetao}$, hence
$\bm c_s=G'(U)-G(U)+G(t)$ with $U=T_s$ and $t=T_{s-1}$. Given $T_{s-1}=t$, the
variable $U-t$ is $\operatorname{Exp}(1)$ by \eqref{eq:arrival}, and $G$ grows at
most polynomially (equation~\eqref{eq:D-uniform}). An integration by parts
therefore gives
\begin{equation}
\Ep_{\thetao}(\bm c_s\mid X^{(1)},\ldots,X^{(s-1)})
=\int_t^\infty e^{-(u-t)}\{G'(u)-G(u)+G(t)\}\,du
=-G(t)+G(t)=0 .
\label{eq:conditional-score-mean}
\end{equation}
Each $\bm c_s=\dot{\boldsymbol\ell}_{s,\thetao}-\dot{\boldsymbol\ell}_{s-1,\thetao}$ is
square integrable (Lemma~\ref{lem:info-cont}), so $(\bm c_s)_{s\ge2}$ is a
square-integrable martingale-difference sequence with respect to
$\mathcal F_s=\sigma(X^{(1)},\ldots,X^{(s)})$. All cross-covariances in
\eqref{eq:score-decomposition} vanish, and
\begin{equation}
I_r(\thetao)=I_1(\thetao)+\sum_{s=2}^r\Ep_{\thetao}(\bm c_s\bm c_s^T)
\succeq I_1(\thetao).
\label{eq:information-decomposition}
\end{equation}
Finally, $I_1(\thetao)>0$ because the three GEV score components are linearly
independent functions of $T_1\sim\operatorname{Exp}(1)$ (Lemma~\ref{lem:I1-pd}). This proves Proposition~\ref{prop:info-pd}.

\subsection{Proof of Theorem~\ref{thm:info-identity}}
\label{sec:proof-identity}

Since the support moves, we cannot differentiate $\int p_{r,\thet}=1$ twice under
the integral sign. Instead we apply the fundamental theorem of calculus along
parameter segments for each fixed $\xv$, and then integrate over $\xv$. Define,
on $S_{r,\thet}$ and as zero elsewhere,
\[
\dot p_{\thet,i}=p_{r,\thet}\,\partial_i\ell_{r,\thet},
\qquad
\ddot p_{\thet,ij}=p_{r,\thet}
\bigl(\partial_{ij}^2\ell_{r,\thet}+\partial_i\ell_{r,\thet}\partial_j\ell_{r,\thet}\bigr),
\]
and fix a compact box $V\subset\Theta_{1/2}$ around $\thetao$ with
$-\bar k\le k\le k_+<1/2$ and $c\le\sigma\le C$ on $V$. Two lemmas, proved in
Appendix~\ref{sec:S-hessian}, play complementary roles.

\begin{lemma}[$L^1$ continuity]
\label{lem:density-derivatives}
The maps $\thet\mapsto\dot p_{\thet,i}$ and
$\thet\mapsto\ddot p_{\thet,ij}$ are continuous from $V$ into $L^1(\R^r)$.
\end{lemma}

\begin{lemma}[Absolute continuity across moving endpoints]
\label{lem:ac-endpoint}
For every strictly ordered $\xv\in\R^r$ and every coordinate segment
$\{\thet+s\bm e_i\}\subset V$, the maps $s\mapsto p_{r,\thet+s\bm e_i}(\xv)$ and
$s\mapsto\dot p_{\thet+s\bm e_i,j}(\xv)$ are absolutely continuous, with derivatives
$\dot p_{\thet+s\bm e_i,i}(\xv)$ and $\ddot p_{\thet+s\bm e_i,ij}(\xv)$ outside a
finite set.
\end{lemma}

\begin{proof}[Proof of Theorem~\ref{thm:info-identity}]
Integrability of $\ddot{\boldsymbol\ell}_{r,\thetao}$ follows from the bound
\[
C(1+T_r)\sum_j\max(1,T_j^{-2k_0})(1+|\log T_j|)^4
\]
in \eqref{eq:hessian-pointwise}, since $T_1^{-2k_0}$ is integrable for $k_0<1/2$.

Let $\thet$ be interior to $V$. By Lemma~\ref{lem:ac-endpoint},
$p_{r,\thet+h\bm e_i}(\xv)-p_{r,\thet}(\xv)=\int_0^h\dot p_{\thet+s\bm e_i,i}(\xv)\,ds$
for every strictly ordered $\xv$. The bounds behind this identity depend on
$\xv$; integrability in $(\xv,s)$ comes instead from Lemma~\ref{lem:density-derivatives},
which gives $\int_0^{|h|}\!\int|\dot p_{\thet+s\bm e_i,i}|\,d\xv\,ds<\infty$. The integrand
is jointly measurable in $(\xv,s)$, being given by an explicit Borel formula on the
Borel set $\{(\xv,s):\xv\in S_{r,\thet+s\bm e_i}\}$ and equal to zero elsewhere. By
Fubini's theorem,
\[
0=\int p_{r,\thet+h\bm e_i}\,d\xv-\int p_{r,\thet}\,d\xv
=\int_0^h\Bigl\{\int\dot p_{\thet+s\bm e_i,i}\,d\xv\Bigr\}ds
\]
for all small $h$, and the inner integral is continuous in $s$. Hence
$\int\dot p_{\thet,i}\,d\xv=0$ for every $\thet$ interior to $V$. Applying the same
argument to $s\mapsto\dot p_{\thet+s\bm e_j,i}$, whose integral is therefore
constant near $\thetao$, gives $\int\ddot p_{\thet,ij}\,d\xv=0$. At $\thet=\thetao$
this is
$\Ep_{\thetao}\{\partial^2_{ij}\ell_{r,\thetao}+\partial_i\ell_{r,\thetao}\partial_j\ell_{r,\thetao}\}=0$.
\end{proof}

\begin{remark}[Where $k_0<1/2$ enters a second time]
In Lemma~\ref{lem:ac-endpoint} the threshold appears once more. At a crossing
$s^*$ with $k^*>0$, the quantity $b=w_{\thet(s)}(x_1)$, which is proportional to
the distance from $x_1$ to the moving endpoint, is of order
$u=|s-s^*|$, $T_1=b^{1/k}$, and each parameter derivative costs a factor $b^{-1}$,
so $|\ddot p_{\thet(s),ij}(\xv)|\le Cu^{1/k_+-3}(1+|\log u|)^4$. Since
$\int_0^1u^{\beta}(-\log u)^m\,du=m!/(\beta+1)^{m+1}$ for $\beta>-1$,
\[
\int_0^1u^{1/k_+-3}(1+|\log u|)^4\,du
=\sum_{m=0}^4\binom4m\,m!\Bigl(\frac{k_+}{1-2k_+}\Bigr)^{m+1},
\]
which is finite exactly when $k_+<1/2$ and blows up as $k_+\uparrow1/2$. The bound
is unbounded in $u$ for $1/3\le k^*<1/2$; only its integrability survives.
\end{remark}

\subsection{Proofs of Corollary~\ref{cor:lan} and Remark~\ref{rem:lan-failure}}
\label{sec:proof-lan}

(i) The parameter set $\{\sigma>0\}$ is open and the model is DQM at $\thetao$ by
Theorem~\ref{thm:dqm}. Theorem~7.2 of \citet{vanderVaart1998} then gives the
expansion, with mean-zero score and covariance $I_r(\thetao)$; the central limit
theorem gives the limit of $\bm\Delta_n$, and Le Cam's first lemma
\citep[Lemma~6.4 and Example~6.5]{vanderVaart1998} gives mutual contiguity. (ii) This is the
local asymptotic minimax theorem \citep[Theorem~8.11]{vanderVaart1998}, whose
assumptions are DQM, nonsingular information (Proposition~\ref{prop:info-pd})
and differentiability of $\psi$.

For Remark~\ref{rem:lan-failure}, let $A_n$ be the event that some
$X_i^{(1)}\ge\omega_0$. Then $P^n_{\thetao}(A_n)=0$, while by
\eqref{eq:dqm-necessity-mass} with $t=a/\sqrt n$,
$p_n:=P_{\thet_n}\{X^{(1)}\ge\omega_0\}\sim(k_0a/\sigma_0)^{1/k_0}n^{-1/(2k_0)}$ and
$P^n_{\thet_n}(A_n)=1-(1-p_n)^n\ge1-e^{-np_n}$. Here $np_n\to\infty$ if $k_0>1/2$
and $np_n\to a^2/(4\sigma_0^2)>0$ if $k_0=1/2$. Hence
$\liminf_nP^n_{\thet_n}(A_n)>0$, which contradicts contiguity
\citep[Section~6.2]{vanderVaart1998}.

\section{Discussion}
\label{sec:discussion}

For $r=1$, Theorem~\ref{thm:dqm} reduces to the result of
\citet{buchersegers2017maximum}, whose moving-support approach we follow. For
$r\ge2$, the Poisson-arrival representation reduces the ordered vector to
exponential spacings, and the proofs of Lemma~\ref{lem:sqrt-smooth} and
Lemma~\ref{lem:ac-endpoint} show that only $X^{(1)}$ (for $k>0$) or $X^{(r)}$ (for
$k<0$) can leave the support. The arrival representation and segmentwise absolute continuity also avoid
unjustified differentiation under moving integrals in the conditional score identity
\eqref{eq:conditional-score-mean} and in the information identity. They should
apply to other fixed-dimensional order-statistic likelihoods whose supports are
described by finitely many smooth endpoint inequalities.

The direct proof of Theorem~\ref{thm:info-identity} is independent of
maximum-likelihood theory. This separation is useful: estimator consistency
\citep[see, e.g.,][for consistency of the maximum likelihood estimator through
its uniform consistency]{AkahiraKashima1994}, score expansions and uniform
Hessian laws require additional arguments when the support moves, and none is
assumed here. The information identity is
therefore available as an input to subsequent estimation theory rather than
as a consequence of it.

The results concern the exact model with a fixed $r$ and independent blocks.
Allowing $r=r_n\to\infty$, selecting $r$ from the data, or accounting for serial
dependence \citep{BucherHaufs2026} or finite-block approximation
\citep{dombry2019maximum} requires further arguments. In the non-regular range
$k_0\ge1/2$, the leading coordinate has the endpoint behavior studied by
\citet{Smith1985}, but the non-regular estimation rates for $r\ge2$ are not
established here. The asymptotic theory of the maximum likelihood estimator,
which uses the results of this paper, is the subject of ongoing work.

\section{Conclusion}
\label{sec:conclusion}

For every fixed $r$, the regularity of the three-parameter $r$GEV model is
governed by the leading maximum: the model is DQM exactly when $k_0<1/2$, as for
block maxima, while each additional order statistic contributes a positive
semidefinite information term. The information identity and local asymptotic
normality hold in the same range and are properties of the model alone, so they
are available as inputs to the asymptotic theory of estimators and tests based on
a fixed number of upper order statistics.

\section*{Declarations}

\subsection*{ORCID}
Jihong Park: 0009-0003-3191-9968; ~~~
Yire Shin: 0000-0003-1297-5430; \\
Jeong-Soo Park: 0000-0002-8460-4869.

\subsection*{Funding}
Jihong Park's research was supported by the BK21 FOUR (Fostering Outstanding
Universities for Research, No.~5120200913674), funded by the Ministry of
Education (MOE, Korea) and the National Research Foundation of Korea (NRF).
Yire Shin's research was supported by the Basic Science Research Program through
the National Research Foundation of Korea (NRF), funded by the Ministry of
Education (RS-2025-25436608).

\subsection*{Conflict of interest}
The authors declare that they have no conflicts of interest.

\subsection*{Data availability}
No data were used. The numerical values in Section~\ref{sec:numerics} are
computed from the formulas given there.

\subsection*{Declaration of generative AI and AI-assisted technologies in the manuscript preparation process}
Generative artificial intelligence tools, namely ChatGPT (OpenAI) and Claude
(Anthropic), were used during this work. They assisted with exploring proof
strategies, drafting and checking mathematical derivations, and editing the manuscript
for clarity. All statements, proofs, numerical results, references and
conclusions were checked by the authors, who take full responsibility for the
content of this article.

\subsection*{Author contributions}
All authors contributed to the conception of the study, the development of the
methodology, and the derivation and verification of the theoretical results.
The first draft of the manuscript was written by Y. Shin and J.-S. Park, and the
other authors commented on and revised previous versions of the manuscript. All
authors read and approved the final manuscript.

\appendix

\section*{Appendices}
Appendix~\ref{sec:S-dqm} proves the sufficiency part of Theorem~\ref{thm:dqm}
and the positive definiteness of $I_1(\thetao)$ used in
Proposition~\ref{prop:info-pd}; Appendix~\ref{sec:S-hessian} proves
Lemmas~\ref{lem:density-derivatives} and~\ref{lem:ac-endpoint}, which are used in
the proof of Theorem~\ref{thm:info-identity}. Throughout,
$\Theta_{1/2}=\{(\mu,\sigma,k):\sigma>0,\ k<1/2\}$.

\section{Differentiability in quadratic mean and the block-maximum information}
\label{sec:S-dqm}

\begin{lemma}[Continuity of the fixed-$r$ information]
\label{lem:info-cont}
The information matrix $I_r(\thet)$ is finite and continuous on $\Theta_{1/2}$.
\end{lemma}

\begin{proof}
The difficulty is that $I_r(\thet)$ is an expectation under $P_{r,\thet}$, so
both the integrand and the measure change with $\thet$. We remove the second
dependence by a change of variables. Under $P_{r,\thet}$ the transformed vector
$(T_{\thet}(X^{(1)}),\ldots,T_{\thet}(X^{(r)}))$ has the parameter-free arrival
law \eqref{eq:arrival}. Moreover, by \eqref{eq:D} and \eqref{eq:joint-score},
the score depends on the observation only through $t_j=T_{\thet}(x_j)$:
\[
\dot{\boldsymbol\ell}_{r,\thet}(\xv)=\bm\psi_{\thet}(t_1,\ldots,t_r),
\qquad
\bm\psi_{\thet}(\bm t)=-\frac r\sigma\bm e_\sigma-t_rD_{\thet}(t_r)
+(1-k)\sum_{j=1}^rD_{\thet}(t_j)-\bm e_k\sum_{j=1}^r\log t_j,
\]
where $D_{\thet}(t)$ denotes \eqref{eq:D} written as a function of $t$; it
involves only $\sigma$, $k$ and $t$. Hence
$I_r(\thet)=\Ep\{\bm\psi_{\thet}(T_1,\ldots,T_r)\bm\psi_{\thet}(T_1,\ldots,T_r)^T\}$,
where $(T_1,\ldots,T_r)$ has the law \eqref{eq:arrival}, which does not
depend on $\thet$. (In particular, $I_r(\thet)$ does not depend on $\mu$.)
Continuity of $I_r$ therefore follows from continuity of
$\thet\mapsto\bm\psi_{\thet}(\bm t)$ for each fixed $\bm t$, together with a
square-integrable envelope for $\sup_{\thet\in V}\|\bm\psi_{\thet}\|$ on compact
sets $V$, which we now construct.
Fix a compact $V\subset\Theta_{1/2}$ and choose constants
$k_+<1/2$, $\bar k<\infty$, and $0<c<C<\infty$ such that
\[
-\bar k\le k\le k_+,
\qquad c\le\sigma\le C,
\qquad \thet\in V.
\]
The apparent singularities at $k=0$ in \eqref{eq:D} are removable.  Writing
$L=\log t$, the integral form of Taylor's formula gives
\begin{equation}
\begin{aligned}
D_\mu&=\sigma^{-1}e^{-kL},\\
D_\sigma&=-\frac{L}{\sigma}\int_0^1e^{-vkL}\,dv,\\
D_k&=-L^2\int_0^1(1-v)e^{-vkL}\,dv,
\end{aligned}
\label{eq:D-integral}
\end{equation}
which at $k=0$ reduce to $\sigma^{-1}$, $-L/\sigma$ and $-L^2/2$.
Since $e^{-vkL}=t^{-vk}\le\max(1,t^{-k})$ for $0\le v\le1$,
\begin{equation}
\|D_{\thet}(t)\|
\le
C_1\max(1,t^{-k})\{1+|\log t|\}^2
\le
C_1\{1+t^{-k_+}+t^{\bar k}\}\{1+|\log t|\}^2,
\label{eq:D-uniform}
\end{equation}
uniformly over $\sigma\in[c,C]$, for every real $k$ in the first inequality
and over $V$ in the second. The envelope is stated in the fixed arrival
variables $0<t_1<\cdots<t_r$, which do not depend on $\thet$. Since
$t_1\le t_j\le t_r$, we have $|\log t_j|\le |\log t_1|+\log(1+t_r)$, and the
definition of $\bm\psi_{\thet}$ gives, for some finite $M$,
\[
\sup_{\thet\in V}\|\bm\psi_{\thet}(\bm t)\|
\le
C_2\{1+t_1^{-k_+}+t_r^M\}
\{1+|\log t_1|^2+\log^2(1+t_r)\}.
\]
We evaluate this envelope at the arrival times $(T_1,\ldots,T_r)$ of
\eqref{eq:arrival}.

To verify square integrability, write $T_r=T_1+\Delta_r$, where
$\Delta_r=T_r-T_1\sim\operatorname{Gamma}(r-1,1)$ is independent of
$T_1\sim\operatorname{Exp}(1)$ (and $\Delta_1=0$). Expanding the square and
using $(T_1+\Delta_r)^q\le C_q(T_1^q+\Delta_r^q)$ reduces the expectation to
products of one-dimensional moments of $T_1$ and $\Delta_r$. The only
singular factor is a power $T_1^{-2k_+}$ times logarithms, which is integrable
because $2k_+<1$; all other moments are finite. Thus the squared envelope is
integrable. Since $\thet\mapsto\bm\psi_{\thet}(\bm t)$ is continuous for each
fixed $\bm t$ by \eqref{eq:D-integral}, the dominated convergence theorem
 completes the proof.
\end{proof}

\begin{lemma}[Smoothness of the square-root density]
\label{lem:sqrt-smooth}
Let $s_{\thet}=p_{r,\thet}^{1/2}$, extended by zero outside $S_{r,\thet}$, and let
$V$ be an open convex set whose closure lies in $\{(\mu,\sigma,k):k<1/3\}$ and
on which $\sigma$ is bounded away from zero and infinity. Then, for every
$\xv\in\R^r$, the map $\thet\mapsto s_{\thet}(\xv)$ is continuously
differentiable on $V$, with
$\nabla s_{\thet}=\tfrac12s_{\thet}\dot{\boldsymbol\ell}_{r,\thet}$ on
$S_{r,\thet}$ and $\nabla s_{\thet}=0$ off it.
\end{lemma}

\begin{proof}
The value $1/3$ is the threshold for pointwise continuous
differentiability of the square-root density; the DQM threshold itself
is $1/2$.
	Fix $\xv\in\R^r$. If $\xv$ is not strictly ordered, $s_{\thet}(\xv)=0$ for
every $\thet$. Otherwise apply Lemma~\ref{lem:zero-extension} with $m=r$,
$g(\thet)=s_{\thet}(\xv)$ and $c=0$; then $O$ is the set of $\thet\in V$ with
$\xv\in S_{r,\thet}$.

On $O$, $\log T_{\thet}(x_j)=-\int_0^{z_j}(1-kv)^{-1}\,dv$ with
$z_j=(x_j-\mu)/\sigma$ is real-analytic in $\thet$, including across $k=0$,
so (i) holds and $\nabla s_{\thet}=\tfrac12s_{\thet}\dot{\boldsymbol\ell}_{r,\thet}$.
For (iii), let $\thet^*\in V\cap\partial O$; then $\pi_j(\thet^*)\ge0$ for all
$j$, with equality for at least one $j$. If $k^*=0$, then
$\pi_j(\thet^*)=\sigma^*>0$ for all $j$, which is impossible; hence $k^*\ne0$.
For $k^*\ne0$, $\pi_j(\thet^*)=\sigma^*-k^*(x_j-\mu^*)$ is strictly monotone
in $x_j$, so the distinct values $x_1>\cdots>x_r$ give distinct values
$\pi_j(\thet^*)$, and exactly one of them vanishes: the smallest, which is
$\pi_1$ if $k^*>0$ and $\pi_r$ if $k^*<0$. In other words, a parameter path
can leave the support only through $x_1$ at an upper endpoint or through
$x_r$ at a lower endpoint. All other $T_{\thet}(x_j)$, and the score terms
involving them, converge to finite positive limits as $\thet\to\thet^*$.
Write $T_j=T_{\thet}(x_j)$.
\begin{itemize}
\item If $k^*>0$, then $\pi_1(\thet^*)=0$ and $T_1\to0$. The other factors
in \eqref{eq:rgev-density} stay bounded, so $s_{\thet}\le CT_1^{(1-k)/2}\to0$.
In \eqref{eq:joint-score} the terms involving $j=1$ are
$(1-k)D_1-\bm e_k\log T_1$ (and also $-T_1D_1$ when $r=1$); by
\eqref{eq:D-uniform} they are bounded by $CT_1^{-k}(1+|\log T_1|)^2$. Since
$\nabla s_{\thet}=\tfrac12 s_{\thet}\dot{\boldsymbol\ell}_{r,\thet}$, the
exponents add: $(1-k)/2-k=(1-3k)/2$. Hence
\[
\|\nabla s_{\thet}(\xv)\|\le CT_1^{(1-3k)/2}(1+|\log T_1|)^2\longrightarrow0,
\]
because $k\to k^*<1/3$.
\item If $k^*<0$, then $\pi_r(\thet^*)=0$ and $T_r\to\infty$. Now
$s_{\thet}\le Ce^{-T_r/2}T_r^{(1-k)/2}\to0$, and by \eqref{eq:D-uniform} the
terms involving $j=r$ are bounded by $CT_r^{1+|k|}(1+\log T_r)^2$. Hence
$\|\nabla s_{\thet}(\xv)\|\le Ce^{-T_r/2}T_r^{3(1+|k|)/2}(1+\log T_r)^2\to0$.
\end{itemize}
The constants may depend on $\xv$, which is allowed in (iii).
Lemma~\ref{lem:zero-extension} now shows that $\thet\mapsto s_{\thet}(\xv)$ is
continuously differentiable on $V$, with $\nabla s_{\thet}=0$ off the support.
The bounds \eqref{eq:D-uniform} used here contain no factor $1/k$, so the
argument covers both signs of $k$ near $k=0$.
\end{proof}

\begin{proof}[Proof of sufficiency in Theorem~\ref{thm:dqm}]
	Let
	\[
	s_{\thet}(\xv)=p_{r,\thet}(\xv)^{1/2},
	\]
	where $s_{\thet}$ is extended by zero outside $S_{r,\thet}$.
	The argument adapts the moving-support proof of
	\citet[Appendix C]{buchersegers2017maximum} to the fixed-$r$ joint density.
	
	\medskip\noindent
	\emph{Case 1: $k_0<1/3$.}
	Choose an open convex neighborhood $V$ of $\thetao$ whose closure
	lies in $\{k<1/3\}$ and has $\sigma$ bounded away from zero and infinity.
	By Lemma~\ref{lem:sqrt-smooth}, $\thet\mapsto s_{\thet}(\xv)$ is
	continuously differentiable on $V$ for every $\xv$.
	Lemma~7.6 of \citet{vanderVaart1998} states that a family is differentiable
	in quadratic mean at $\thetao$ if, for every $\xv$, the map
	$\thet\mapsto s_{\thet}(\xv)$ is continuously differentiable near $\thetao$
	and the information matrix is finite and continuous there. The first
	condition has just been verified and the second is
	Lemma~\ref{lem:info-cont}. Hence the family is differentiable in quadratic
	mean at $\thetao$.

	\medskip\noindent
	\emph{Case 2: $1/3\le k_0<1/2$.}
	Let $t_n\downarrow0$, $\bm h_n\to\bm h$, and put
	\[
	\thet_n=\thetao+t_n\bm h_n.
	\]
	It suffices to prove
	\begin{equation}
	\int
	\left[
	\frac{s_{\thet_n}(\xv)-s_{\thetao}(\xv)}{t_n}
	-
	\frac12
	\bm h^T
	\dot{\boldsymbol\ell}_{r,\thetao}(\xv)
	s_{\thetao}(\xv)
	\right]^2
	d\xv
	\longrightarrow0.
	\label{eq:dqm-target}
	\end{equation}
	
	For $u\in[0,1]$, let
	\[
	\thet_{n,u}
	=
	\thetao+ut_n\bm h_n.
	\]
	Since $k_0>0$, all sufficiently close parameter values have a finite
	upper endpoint
	\[
	\omega(\thet)=\mu+\frac{\sigma}{k}.
	\]
	Define
	\[
	\alpha_n
	=
	\inf_{0\le u\le1}\omega(\thet_{n,u}),
	\qquad
	\beta_n
	=
	\sup_{0\le u\le1}\omega(\thet_{n,u}).
	\]
	Smoothness of $\omega$ in a neighborhood of $\thetao$ gives
	\begin{equation}
	|\alpha_n-\omega_0|
	+
	|\beta_n-\omega_0|
	=
	O(t_n),
	\qquad
	\beta_n-\alpha_n
	=
	O(t_n),
	\label{eq:dqm-endpoint-width}
	\end{equation}
	where $\omega_0=\omega(\thetao)$.
	
	On the region $\{x_1<\alpha_n\}$, the entire line segment
	$\{\thet_{n,u}:0\le u\le1\}$ has common support.  Hence
	\[
	\frac{s_{\thet_n}(\xv)-s_{\thetao}(\xv)}{t_n}
	=
	\int_0^1
	\nabla_{\thet}s_{\thet_{n,u}}(\xv)^T
	\bm h_n\,du.
	\]
	Using Cauchy--Schwarz, Fubini's theorem, and
	\[
	\nabla_{\thet}s_{\thet}
	=
	\frac12s_{\thet}
	\dot{\boldsymbol\ell}_{r,\thet}
	\]
	on the interior of the support, we obtain
	\begin{align}
	&\int_{\{x_1<\alpha_n\}}
	\left\{
	\frac{s_{\thet_n}-s_{\thetao}}{t_n}
	\right\}^2d\xv
	\nonumber\\
	&\qquad\le
	\frac14
	\int_0^1
	\bm h_n^T
	I_r(\thet_{n,u})
	\bm h_n\,du.
	\label{eq:dqm-common-support}
	\end{align}
	By Lemma~\ref{lem:info-cont}, the right-hand side converges to
	\[
	\frac14\bm h^T I_r(\thetao)\bm h.
	\]
	
	It remains to control the strip $\{\alpha_n\le x_1\le\beta_n\}$ in which the
	support boundary moves. Both endpoints $\omega(\thetao)$ and
	$\omega(\thet_n)$ lie between $\alpha_n$ and $\beta_n$. On
	$\{x_1>\beta_n\}$ both zero-extended square-root densities vanish, so this
	region contributes nothing. On the strip we use $(a-b)^2\le2a^2+2b^2$ and
	$\int_A s_{\thet}^2\,d\xv=P_{\thet}(A)$; thus
	\begin{align}
	&\int_{\{\alpha_n\le x_1\le\beta_n\}}
	\left\{
	\frac{s_{\thet_n}-s_{\thetao}}{t_n}
	\right\}^2d\xv
	\nonumber\\
	&\qquad\le
	2t_n^{-2}
	\left[
	P_{\thet_n}\{\alpha_n\le X^{(1)}\le\beta_n\}
	+
	P_{\thetao}\{\alpha_n\le X^{(1)}\le\beta_n\}
	\right].
	\label{eq:dqm-boundary-strip}
	\end{align}
	
	The first marginal of the $r$GEV distribution is the ordinary GEV
	distribution.  More generally, if
	$\widetilde\thet_n\to\thetao$, with shape component
	$\widetilde k_n$, and $d_n=O(t_n)$, then
	\begin{align}
	&P_{\widetilde\thet_n}
	\left\{
	X^{(1)}
	\ge
	\omega(\widetilde\thet_n)-d_n
	\right\}
	\nonumber\\
	&\qquad=
	1-
	\exp\left[
	-
	\left\{
	\frac{\widetilde k_n d_n}{\widetilde\sigma_n}
	\right\}^{1/\widetilde k_n}
	\right]
	\le
	C d_n^{\,1/\widetilde k_n}.
	\label{eq:dqm-endpoint-tail}
	\end{align}
	(The equality evaluates the GEV distribution function at
	$\omega(\widetilde\thet_n)-d_n$, where
	$T_{\widetilde\thet_n}=\{\widetilde k_nd_n/\widetilde\sigma_n\}^{1/\widetilde k_n}$;
	the inequality uses $1-e^{-x}\le x$.)
	Apply \eqref{eq:dqm-endpoint-tail} twice with $d_n=\beta_n-\alpha_n$, which
	is $O(t_n)$ by \eqref{eq:dqm-endpoint-width}: once with
	$\widetilde\thet_n=\thet_n$, since
	$\{X^{(1)}\ge\alpha_n\}\subset\{X^{(1)}\ge\omega(\thet_n)-d_n\}$, and once
	with $\widetilde\thet_n=\thetao$ in the same way. Since $k_0<1/2$, there
	exists $\delta>0$ such that, eventually,
	\[
	\frac1{\widetilde k_n}>2+\delta,
	\qquad\text{so}\qquad
	t_n^{-2}d_n^{\,1/\widetilde k_n}=O(t_n^{\delta})\longrightarrow0.
	\]
	Hence the right-hand side of \eqref{eq:dqm-boundary-strip} is $o(1)$. This
	is exactly where the condition $k_0<1/2$ enters: the probability mass in a
	strip of width $O(t_n)$ below the endpoint must be $o(t_n^2)$.
	
	For almost every fixed $\xv$, the difference quotient in
	\eqref{eq:dqm-target} converges pointwise. Indeed, if $x_1<\omega_0$, then
	$x_1<\alpha_n$ eventually and the quotient is an average of derivatives
	along a segment inside the support; if $x_1>\omega_0$, both terms vanish
	eventually; and $\{x_1=\omega_0\}$ is a null set. The limit is
	\[
	\frac12
	\bm h^T
	\dot{\boldsymbol\ell}_{r,\thetao}(\xv)
	s_{\thetao}(\xv).
	\]
	Moreover, \eqref{eq:dqm-common-support} and the vanishing boundary
	strip imply
	\[
	\limsup_{n\to\infty}
	\int
	\left\{
	\frac{s_{\thet_n}-s_{\thetao}}{t_n}
	\right\}^2d\xv
	\le
	\frac14
	\bm h^T I_r(\thetao)\bm h,
	\]
	which is the squared $L^2$ norm of the pointwise limit. By Fatou's lemma
	the reverse inequality holds for the $\liminf$, so the squared norms
	converge. Proposition~2.29 of \citet{vanderVaart1998} states that
	convergence almost everywhere together with
	$\limsup_n\|f_n\|_{L^p}\le\|f\|_{L^p}<\infty$ implies
	$\|f_n-f\|_{L^p}\to0$. With $p=2$, this upgrades the pointwise convergence
	to \eqref{eq:dqm-target}; the moving strip has already been shown to carry
	vanishing squared mass.  Hence the model is differentiable in
	quadratic mean also for $1/3\le k_0<1/2$.

	The case $k_0\ge1/2$ (necessity) is proved in Section~\ref{sec:proof-dqm}. This completes the proof.
\end{proof}

\begin{lemma}[Positive definiteness of the block-maximum information]
\label{lem:I1-pd}
If $k_0<1/2$, then $I_1(\thetao)$ is positive definite.
\end{lemma}

\begin{proof}
Ordinary-GEV information calculations go back to
\citet{prescott1980maximum}; we give an explicit score-based
verification here, rather than assume invertibility.
	By Lemma~\ref{lem:info-cont}, the ordinary GEV score has finite second
	moments when $k_0<1/2$, so it suffices to establish linear independence
	of its three components in $L^2(P_{\thetao})$.
	
	Let
	\[
	T=T_{\thetao}(X^{(1)})\sim\operatorname{Exp}(1).
	\]
	For $k\ne0$, the ordinary GEV score components, written as functions
	of $t>0$, are
	\[
	S_\mu(t)
	=
	\frac{(1-k-t)t^{-k}}{\sigma},
	\]
	\[
	S_\sigma(t)
	=
	-\frac1{\sigma}
	+
	\frac{1-k-t}{k\sigma}
	(t^{-k}-1),
	\]
	and
	\[
	S_k(t)
	=
	-\log t
	+
	(1-k-t)
	\left[
	-\frac{\log t}{k}
	-\frac{t^{-k}-1}{k^2}
	\right],
	\]
	with the corresponding continuous limits at $k=0$.
	
	Suppose that, for some constants $a,b,c$,
	\begin{equation}
	aS_\mu(T)+bS_\sigma(T)+cS_k(T)=0
	\qquad\text{a.s.}
	\label{eq:score-linear-combination}
	\end{equation}
	Since $T$ has a strictly positive density on $(0,\infty)$ and the
	left-hand side is continuous in $t$, the identity in
	\eqref{eq:score-linear-combination} must hold for every $t>0$. The
	left-hand side is moreover infinitely differentiable in $t>0$, so all of
	its derivatives vanish identically. We use those at $t=1$, where
	$t^{-k}-1$ and $\log t$ vanish and the expansions are simplest; three
	linearly independent equations in $(a,b,c)$ then force $a=b=c=0$.
	
	The derivative matrices below follow from the Taylor expansions of
$t^{-k}$ and $\log t=(t-1)-(t-1)^2/2+(t-1)^3/3+O((t-1)^4)$ at $t=1$. With
$e=t-1$, they give
\[
\begin{aligned}
\sigma S_\mu(t)&=-k+(k^2-1)e+\Bigl\{k-\frac{k^2(k+1)}2\Bigr\}e^2\\
&\quad+\Bigl\{\frac{k^2(k+1)(k+2)}6-\frac{k(k+1)}2\Bigr\}e^3+O(e^4),\\
\sigma S_\sigma(t)&=-1+ke+\Bigl\{1-\frac{k(k+1)}2\Bigr\}e^2
+\Bigl\{\frac{k(k+1)(k+2)}6-\frac{k+1}2\Bigr\}e^3+O(e^4),\\
S_k(t)&=-e+\frac{1+k}2e^2+\frac{1-3k-k^2}6e^3+O(e^4).
\end{aligned}
\]
The $m$th derivative at $t=1$ is $m!$ times the coefficient of $e^m$.
For $k\ne-1$, evaluate this identity and its first two derivatives
	with respect to $t$ at $t=1$.  The resulting linear system for
	$(a,b,c)^T$ has coefficient matrix
	\[
	\begin{pmatrix}
	-k/\sigma
	&
	-1/\sigma
	&
	0
	\\[1ex]
	(k^2-1)/\sigma
	&
	k/\sigma
	&
	-1
	\\[1ex]
	k(-k^2-k+2)/\sigma
	&
	(-k^2-k+2)/\sigma
	&
	k+1
	\end{pmatrix}.
	\]
	Expanding along the first row, its determinant is
	\[
	\frac1{\sigma^2}\bigl[-k\{k(k+1)+(2-k-k^2)\}+(k^2-1)(k+1)+k(2-k-k^2)\bigr]
	=\frac{-2k+(k-1)}{\sigma^2}
	=-\frac{k+1}{\sigma^2},
	\]
	which is nonzero for $k\ne-1$.  Hence
	\[
	a=b=c=0.
	\]
	
	When $k=-1$, the preceding determinant vanishes, but using instead the
	zeroth, first, and third derivatives at $t=1$ gives the coefficient
	matrix
	\[
	\begin{pmatrix}
	1/\sigma
	&
	-1/\sigma
	&
	0
	\\[1ex]
	0
	&
	-1/\sigma
	&
	-1
	\\[1ex]
	0
	&
	0
	&
	3
	\end{pmatrix},
	\]
	whose determinant is
	\[
	-\frac{3}{\sigma^2}\ne0.
	\]
	Thus $a=b=c=0$ also for $k=-1$.
	
	At $k=0$ the score components are the continuous limits of those for
	$k\ne0$, and their derivatives at $t=1$ are continuous in $k$. Hence the
	general coefficient matrix applies at $k=0$ as well, with determinant
	$-1/\sigma^2\ne0$; only $k=-1$ needs the separate treatment above.
	Hence the three ordinary GEV score components are linearly independent
	for every $k<1/2$.
	
	Therefore, for every nonzero vector $\bm a\in\mathbb R^3$,
	\[
	\bm a^T I_1(\thetao)\bm a
	=
	\Ep_{\thetao}
	\left[
	\left\{
	\bm a^T
	\dot{\boldsymbol\ell}_{1,\thetao}(X^{(1)})
	\right\}^2
	\right]
	>0.
	\]
	Thus
	\[
	I_1(\thetao)>0.
	\]
\end{proof}

\section{Hessian bounds and the information identity}
\label{sec:S-hessian}

For $\xv\in S_{r,\thet}$ write $T_j=T_{\thet}(x_j)$, $L_j=\log T_j$,
$D_j=D_{\thet}(x_j)$ and
\[
D^{(2)}_j=D^{(2)}_{\thet}(x_j),\qquad
D^{(2)}_{\thet}(x)=\nabla^2_{\thet}\log T_{\thet}(x).
\]
Differentiating the score \eqref{eq:joint-score} once more, and using
$\nabla_{\thet}T_r=T_rD_r$, gives
\begin{equation}
\ddot{\boldsymbol\ell}_{r,\thet}(\xv)
=\frac r{\sigma^2}\bm e_\sigma\bm e_\sigma^T
-T_r\bigl(D_rD_r^T+D^{(2)}_r\bigr)
+(1-k)\sum_{j=1}^rD^{(2)}_j
-\sum_{j=1}^r\bigl(\bm e_kD_j^T+D_j\bm e_k^T\bigr).
\label{eq:hessian}
\end{equation}
The first term comes from $-r\log\sigma$, the second from $-T_r$, the third
from $(1-k)\sum_j\log T_j$, and the last from differentiating the factor
$(1-k)$ in the third term of \eqref{eq:joint-score} and the factor $\bm e_k$
in its fourth term.

\begin{lemma}[Second derivatives of $\log T_{\thet}$]
\label{lem:hess-D}
Let $0<c\le\sigma\le C$ and $|k|\le\bar k$. For every $x$ with
$w_{\thet}(x)>0$, put $L=\log T_{\thet}(x)$ and
$\kappa=\max\{1,w_{\thet}(x)^{-1}\}$. Then, with constants depending only on
$c$, $C$ and $\bar k$,
\[
\|D_{\thet}(x)\|\le C\,\kappa\,(1+|L|)^2,
\qquad
\|D^{(2)}_{\thet}(x)\|\le C\,\kappa^2(1+|L|)^4 .
\]
\end{lemma}

\begin{proof}
We use the integral representation \eqref{eq:D-integral}, which is valid
for all $k$, including $k=0$. Since $e^{-kL}=w_{\thet}(x)^{-1}$, every
factor $e^{-vkL}=w_{\thet}(x)^{-v}$ with $0\le v\le1$ is at most $\kappa$.
This gives the bound on $D_{\thet}$, as in \eqref{eq:D-uniform}.

For the second derivatives, differentiate each component of
\eqref{eq:D-integral} and use $\nabla_{\thet}L=D_{\thet}(x)$ and
$\nabla_{\thet}(kL)=L\bm e_k+kD_{\thet}(x)$, so that
$\nabla_{\thet}e^{-vkL}=-v\,e^{-vkL}\{L\bm e_k+kD_{\thet}(x)\}$. Writing
$D=D_{\thet}(x)$, we obtain
\[
\begin{aligned}
\nabla D_\mu&=-\sigma^{-2}e^{-kL}\bm e_\sigma-\sigma^{-1}e^{-kL}(L\bm e_k+kD),\\
\nabla D_\sigma&=\frac{L}{\sigma^2}\int_0^1e^{-vkL}dv\,\bm e_\sigma
-\frac1\sigma\int_0^1e^{-vkL}dv\,D
+\frac L\sigma\int_0^1v\,e^{-vkL}dv\,(L\bm e_k+kD),\\
\nabla D_k&=-2L\int_0^1(1-v)e^{-vkL}dv\,D
+L^2\int_0^1(1-v)v\,e^{-vkL}dv\,(L\bm e_k+kD).
\end{aligned}
\]
Differentiation under the integral sign is legitimate here because the
integrands are smooth on the compact interval $[0,1]$. Each term on the
right is a product of at most two factors bounded by $\kappa$, a power of
$|L|$ of order at most four, and constants. Hence
$\|D^{(2)}_{\thet}(x)\|\le C\kappa^2(1+|L|)^4$.
\end{proof}

Combining \eqref{eq:hessian} with Lemma~\ref{lem:hess-D}, and writing
$\kappa_j=\max\{1,w_{\thet}(x_j)^{-1}\}$, we obtain the pointwise bound
\begin{equation}
\|\ddot{\boldsymbol\ell}_{r,\thet}(\xv)\|
\le C\,(1+T_r)\sum_{j=1}^r\kappa_j^2(1+|L_j|)^4,
\qquad\xv\in S_{r,\thet},
\label{eq:hessian-pointwise}
\end{equation}
uniformly over $\thet$ in compact subsets of $\{\sigma>0\}$.

We now prove Lemmas~\ref{lem:density-derivatives} and~\ref{lem:ac-endpoint},
with $\dot p_{\thet,i}$, $\ddot p_{\thet,ij}$ and $V$ as defined
there.

\begin{proof}[Proof of Lemma~\ref{lem:density-derivatives}]
Under the inverse arrival transformation
\[
x_{\thet,j}(t)=\mu+\sigma(1-t_j^k)/k
\]
(with limit $\mu-\sigma\log t_j$ at $k=0$), the law of
$t=(t_1,\ldots,t_r)$ is independent of $\thet$ and has
density $e^{-t_r}$ on $0<t_1<\cdots<t_r$.
The integral representations in \eqref{eq:D-integral}
and Lemma~\ref{lem:hess-D} imply that on $V$ the
transformed squared score and absolute Hessian
are bounded by
\[
C(1+t_r)^M(1+t_1^{-2k_+})
(1+|\log t_1|)^M(1+\log(1+t_r))^M
\]
for some $M<\infty$ and $0<k_+<1/2$.
To make the uniformity explicit, on $V$ choose
$-\bar k\le k\le k_+<1/2$ and $c\le\sigma\le C$.
In arrival coordinates $w_{\thet}(x_{\thet,j}(t))=t_j^k$,
so $\kappa_j=\max(1,t_j^{-k})$. If $0\le k\le k_+$,
then $\kappa_j\le 1+t_1^{-k_+}$; if
$-\bar k\le k<0$, then
$\kappa_j\le 1+t_r^{\bar k}$.
The latter positive powers, together with the $t_r$
factors in the score and Hessian, are absorbed into
$(1+t_r)^M$ by taking $M$ sufficiently large (for
example, $M\ge 4+2\bar k$ for the polynomial factors).
Moreover $|\log t_j|\le |\log t_1|+\log(1+t_r)$.
This is integrable under the arrival law, since
$t_1\sim\operatorname{Exp}(1)$ and $t_r-t_1$ is
independent of $t_1$ with a Gamma law (degenerate
at zero if $r=1$). The transformed score and Hessian
are continuous in $\thet$ for each fixed $t$.
Dominated convergence therefore gives convergence
of the integrals of the positive and negative parts
of $\dot p_{\thet,i}$ and $\ddot p_{\thet,ij}$.
Their zero extensions also converge almost everywhere
in $x$ along every sequence $\thet_m\to\thet\in V$: a point
in the interior of $S_{r,\thet}$ lies in $S_{r,\thet_m}$ for
all large $m$, where continuity applies; a point outside
the closure of $S_{r,\thet}$ lies outside $S_{r,\thet_m}$ for
all large $m$; and the boundary of $S_{r,\thet}$ is a Lebesgue
null set. Scheff\'e's lemma, in the form of
\citet[Proposition~2.29]{vanderVaart1998}, applied to the
positive and negative parts separately, proves
the asserted $L^1$ continuity.
\end{proof}
\begin{proof}[Proof of Lemma~\ref{lem:ac-endpoint}]
Fix a strictly ordered $\xv$ and a coordinate segment
$\thet(s)=\thet+s\bm e_i\subset V$. Put
$\pi_j(s)=\sigma(s)-k(s)(x_j-\mu(s))$.
Each $\pi_j$ is affine along this coordinate segment. It cannot
vanish identically: in the location direction its slope is $k$
(and $k=0$ gives $\pi_j=\sigma>0$); in the scale direction its
slope is $1$; in the shape direction its slope is $-(x_j-\mu)$
(and $x_j=\mu$ gives $\pi_j=\sigma>0$).
Consequently the segment has at most $r$ boundary crossings.
At any crossing $s_*$, exactly one $\pi_j$ vanishes, since
$k(s_*)\ne0$ and the $x_j$ are distinct. The active coordinate
is $j=1$ when $k(s_*)>0$ and $j=r$ when $k(s_*)<0$.
The nonvanishing factors in the density and its first two
parameter derivatives are bounded locally (with constants
allowed to depend on $\xv$). Since the active $\pi_j$ has a
simple zero and $\sigma$ stays bounded away from zero,
there are $c_1,C_1>0$ such that, on the supported side,
\[
 c_1u\le b(s):=w_{\thet(s)}(x_j)\le C_1u,
 \qquad u=|s-s_*|.
\]
If $k(s_*)>0$, shrink the neighborhood so that
$0<k(s)\le k_+<1/2$, $k(s)\ge k(s_*)/2$, and $b(s)\le1/2$.
Writing $t=b^{1/k(s)}$ and $L=\log t$ gives
$|L|\le 2|\log b|/k(s_*)$.
Here the active coordinate is $x_1$, and
\[
p_{r,\thet(s)}(\xv)=\sigma(s)^{-r}\exp\{-T_r\}\,T_1^{1-k(s)}\prod_{j=2}^rT_j^{1-k(s)},
\qquad T_1=b^{1/k(s)},
\]
where $T_j=T_{\thet(s)}(x_j)$. The factors with $j\ge2$ (absent if $r=1$) are
bounded, because $x_2,\ldots,x_r$ stay in the interior of their supports, and
$e^{-T_r}\le1$; moreover $T_1^{1-k(s)}=b^{1/k(s)-1}$. The density formula, the
score formula and Lemma~\ref{lem:hess-D} then yield
\[
 p_{r,\thet(s)}(\xv)\le C b^{1/k(s)-1},\qquad
 \|\dot{\boldsymbol\ell}_{r,\thet(s)}(\xv)\|
 \le C b^{-1}(1+|\log b|)^2,
\]
\[
 \|\ddot{\boldsymbol\ell}_{r,\thet(s)}(\xv)\|
 \le C b^{-2}(1+|\log b|)^4.
\]
In particular, the second density derivative is bounded by
\[
\begin{aligned}
 |\ddot p_{\thet(s),ij}(\xv)|
 &=p_{r,\thet(s)}(\xv)
   |\partial^2_{ij}\ell_{r,\thet(s)}(\xv)
    +\partial_i\ell_{r,\thet(s)}(\xv)
     \partial_j\ell_{r,\thet(s)}(\xv)|\\
 &\le p_{r,\thet(s)}(\xv)
  \bigl(\|\ddot{\boldsymbol\ell}_{r,\thet(s)}(\xv)\|
       +\|\dot{\boldsymbol\ell}_{r,\thet(s)}(\xv)\|^2\bigr)\\
 &\le C b^{1/k(s)-3}(1+|\log b|)^4.
\end{aligned}
\]
The factor $p_{r,\thet(s)}$ multiplies \emph{both} terms
in the middle line. Since $1/k(s)\ge1/k_+$, and
$b(s)\asymp u$, this implies
\[
 |\ddot p_{\thet(s),ij}(\xv)|
 \le C u^{1/k_+-3}(1+|\log u|)^4.
\]
Likewise,
\[
 p_{r,\thet(s)}(\xv)\le C u^{1/k_+-1},\qquad
 |\dot p_{\thet(s),i}(\xv)|\le
 C u^{1/k_+-2}(1+|\log u|)^2.
\]
Both zero-extended functions tend to zero at the crossing,
and the derivative of $q$ is locally integrable, since
\[
 \int_0^\delta u^{1/k_+-3}(1+|\log u|)^4\,du<\infty
 \qquad(k_+<1/2).
\]
This also covers variation of the shape parameter: the
additional logarithmic terms arising from differentiating
$1/k(s)$ are included in the displayed bounds.

If $k(s_*)<0$, choose a neighborhood on which
$-\bar k\le k(s)\le k(s_*)/2<0$ and $b(s)\le 1/2$.
The same coordinatewise crossing calculation gives
$b(s)\asymp u$ (in the shape direction the supported side
is $k(s)>k(s_*)$). Set
$T=b(s)^{-1/|k(s)|}\to\infty$. Then
$\kappa=b^{-1}=T^{|k(s)|}\le T^{\bar k}$ and
$|\log T|=\log T$. The other observation coordinates
remain in the interior of their supports. Here the active
coordinate is $x_r$, so $T_r=T$, and
\[
p_{r,\thet(s)}(\xv)=\sigma(s)^{-r}e^{-T}\,T^{1-k(s)}\prod_{j=1}^{r-1}T_j^{1-k(s)}.
\]
Since $1-k(s)=1+|k(s)|\le1+\bar k$ and $T\ge1$, the factor $T^{1-k(s)}$ is at most
$(1+T)^{1+\bar k}$. The factors with $j\le r-1$ (absent if $r=1$) are bounded,
because $x_1,\ldots,x_{r-1}$ stay in the interior of their supports; one may
also use $T_j\le T$, which only increases the power of $1+T$ below. In the score
\eqref{eq:joint-score}, the term $T_rD_r$ is bounded by $CT^{1+\bar k}(1+\log T)^2$,
because $\kappa_r=b^{-1}=T^{|k(s)|}$ and $|L_r|=\log T$ in Lemma~\ref{lem:hess-D};
the remaining terms are of lower order. Hence, by the density, score and Hessian
formulas,
\[
\begin{aligned}
 p_{r,\thet(s)}(\xv)&\le Ce^{-T}(1+T)^{1+\bar k},\\
 \|\dot{\boldsymbol\ell}_{r,\thet(s)}(\xv)\|
   &\le C(1+T)^{1+\bar k}(1+\log T)^2,\\
 \|\ddot{\boldsymbol\ell}_{r,\thet(s)}(\xv)\|
  +\|\dot{\boldsymbol\ell}_{r,\thet(s)}(\xv)\|^2
   &\le C(1+T)^{2+2\bar k}(1+\log T)^4.
\end{aligned}
\]
Consequently, with $N=3+3\bar k$ and
$C_N=\sup_{T\ge1}e^{-T/2}(1+T)^N(1+\log T)^4<\infty$,
\[
 |p_{r,\thet(s)}(\xv)|+\sum_i|\dot p_{\thet(s),i}(\xv)|
 +\sum_{i,j}|\ddot p_{\thet(s),ij}(\xv)|
 \le Ce^{-T}(1+T)^N(1+\log T)^4
 \le CC_Ne^{-T/2}.
\]
Thus the zero extensions and their derivatives vanish at
the crossing and the derivatives are locally integrable.
At $k(s_*)=0$ no fixed finite observation lies on a support
boundary, because $\pi_j(s_*)=\sigma(s_*)>0$.

On each of the finitely many open subintervals between
crossings, the zero extensions are either smooth or identically
zero. The preceding bounds give continuity at every crossing
and integrability of the derivatives. Applying the fundamental
theorem of calculus on compact subintervals, passing to the
endpoints by absolute integrability, and adding the resulting
identities proves absolute continuity on the entire original
coordinate segment, with the stated derivatives almost everywhere.
\end{proof}

\bibliographystyle{spbasic}
\bibliography{rGEV_A_arXiv_v3}

\end{document}